\documentclass[a4paper,reqno]{amsart}
\usepackage{amssymb}
\usepackage{latexsym}
\usepackage{amsmath}
\usepackage{euscript}
\usepackage{graphics,color}
\usepackage[all]{xy}
\usepackage{cases}
\usepackage[margin=3cm]{geometry}
\usepackage{MnSymbol} 

\newcommand{\be}{\begin{equation}}
\newcommand{\ee}{\end{equation}}
\newcommand{\ba}{\begin{eqnarray}}
\newcommand{\ea}{\end{eqnarray}}
\newcommand{\baa}{\begin{eqnarray*}}
\newcommand{\eaa}{\end{eqnarray*}}
\newcommand{\bb}{\begin{thebibliography}}
\newcommand{\eb}{\end{thebibliography}}

\newcommand{\bi}[1]{\bibitem{#1}}
\newcommand{\lab}[1]{\label{#1}}
\newcommand{\re}[1]{(\ref{#1})}

\newcounter{my}
\newcommand{\he}%
   {\stepcounter{equation}\setcounter{my}%
   {\value{equation}}\setcounter{equation}0%
   }%
\newcommand{\she}%
   {\setcounter{equation}{\value{my}}%
    }%

\renewcommand\t{\tilde}

\newcommand{\olsi}[1]{\,\overline{\!{#1}}} 

\newtheorem{pr}{Proposition}

\newtheorem{theorem}{Theorem}[section]

\theoremstyle{definition}

\newtheorem{remark}[theorem]{Remark}

\numberwithin{equation}{section}

\begin{document}

\title[sieved Jacobi bispectrality]{Bispectrality of the sieved Jacobi polynomials}

\author{Luc Vinet}

\address{IVADO, Centre de Recherches Math\'ematiques and Département de physique, Universit\'e de Montr\'eal, P.O. Box 6128, Centre-ville Station,
Montr\'eal (Qu\'ebec), H3C 3J7, Canada}

\author{Alexei Zhedanov}

\address{School of Mathematics, Renmin University of China, Beijing 100872, China}


\begin{abstract}
It is shown that the CMV Laurent polynomials associated to the sieved Jacobi polynomials on the unit circle satisfy an eigenvalue equation with respect to a first order differential operator of Dunkl type. Using this result, the sieved Jacobi polynomials on the real line are found to be eigenfunctions of a Dunkl differential operator of second order. Eigenvalue equations for the sieved ultraspherical polynomials of the first and second kind are obtained as special cases. These results mean that the sieved Jacobi polynomials (either on the unit circle or on the real line) are bispectral.

\end{abstract}

\maketitle

\section{Introduction} \label{sect:1}
\setcounter{equation}{0}
The main goal of this paper is to construct an eigenvalue equation for the sieved Jacobi polynomials on the real line. We achieve this objective by working in the CMV framework and first obtaining the eigenvalue equation for the Laurent polynomials that are associated to the sieved Jacobi polynomials that are orthogonal on the unit circle. 

A key feature of the CMV formalism is that the recurrence relations of OPUC are presented in terms of ordinary eigenvalue problems for Laurent polynomials. In this respect the CMV approach to OPUC has similarities with the description of PRL. For the latter, the recurrence relation can be set up as an eigenvalue equation for some tridiagonal (Jacobi) matrix. In contrast, for OPUC the recurrence relation involves a five-diagonal matrix (with a specific Zig-Zag structure) which is unitary. This suggests calling certain OPUC CMV-bispectral, if there exists a dual eigenvalue equation involving an operator acting on the argument of the Laurent polynomials in addition to the five-term recurrence eigenvalue relation. In this way, bispectrality for OPUC  becomes similar to how it applies for PRL. 

The sieved ultraspherical polynomials were first introduced in the seminal paper \cite{AAA}. These polynomials appear as nontrivial limits of the q-ultraspherical polynomials \cite{KLS} when the parameter $q$ tends to roots of unity. Geronimo and Van Assche have shown \cite{GVA} that the generic sieved polynomials can be constructed from mappings of the regular (not sieved) orthogonal polynomials. There is also an approach to these same polynomials which is based on the sieved orthogonal polynomials on the unit circle (OPUC) \cite{Ismail_Li} that have simpler characterizations than those of the corresponding sieved orthogonal polynomials on the real line (PRL).   

Although many useful properties of the sieved polynomials have already been established, there remains a significant lacuna in this respect, namely the eigenvalue problem for these polynomials is missing. On the one hand, the q-ultraspherical polynomials \cite{KLS} are known to be eigenfunctions of a definite q-difference operator; on the other hand, an eigenvalue equation for the sieved ultraspherical polynomials is still unknown. These sieved polynomials were shown in \cite{BIW} to obey a second order differential equation but it should be noted that this equation cannot be presented in the form of a second order differential operator having the sieved ultraspherical polynomials as eigenfunctions. Nevertheless, one would expect the sieved ultraspherical polynomials to be solutions of an eigenvalue problem because they appear as limits of ``classical" polynomials that have this property. 

The main result of this paper is the identification of this eigenvalue equation. We first find the one for the Laurent polynomials associated to the sieved Jacobi OPUC, and then, using the Szeg\H{o} map from the unit circle to an interval, we obtain the corresponding eigenvalue equation for the sieved Jacobi polynomials on the real line. 

It turns out that the operators intervening in these eigenvalue equations are differential operators of Dunkl type. This means that they contain reflection operators with respect to the complex argument $z$: $f(z)  \to f\left(q z^{-1} \right)$, where $q$ is a root of unity. The appearance of such operators is rather natural. Indeed, given that the sieved ultraspherical polynomials arise as limits when $q$ goes to roots of unity in the $q$-ultraspherical ones, the presence of these roots in the operators defining the equations is not overly surprising. 

The paper is organized as follows. In Section \ref{sect:2}, we briefly recall the main properties of OPUC in general, and in Section \ref{sect:3}, of the special Jacobi OPUC family. In Section \ref{sect:4}, we consider the sieved Jacobi OPUC in the CMV formalism and present an eigenvalue equation for the associated Laurent polynomials. This will feature a Dunkl first order differential operator, denoted by $L(N)$, involving the reflections attached to the powers of a primitive $N$-th root of unity. The algebraic relations verified by this Dunkl operator in conjunction with the reflections and the generators of the dihedral rotations is the object of Section \ref{sect:5}. The eigenvalue equations for the sieved Jacobi PRL of the first and second kind are derived in Section \ref{sect:6} and feature operators $H(N)$ or $\tilde{H}(N)$ built from quadratic expression in $L(N)$. Operators that commutes with $H(N)$ or $\tilde{H}(N)$ are identified in Section \ref{sect:7} and shown to entail alternate eigenvalue equations. The eigenvalue problems for the generalized ultraspherical polynomials and the sieved ultraspherical polynomials of the first and second type are discussed as special cases in Section \ref{sect:8}. Concluding remarks will be found in Section \ref{sect:9}.

\section{OPUC, CMV and Szeg\H{o} map} \label{sect:2}
\setcounter{equation}{0}
We recall in this section basic facts about orthogonal polynomials on the unit circle and their maps to orthogonal polynomials on the real line.
The OPUC $\Phi_n(z)=z^n + O(z^{n-1})$ are monic polynomials defined by the recurrence relation 
\be
\Phi_{n+1}(z) = z \Phi_n(z) - \olsi a_n \Phi_n^*(z), \quad \Phi_0(z) =1 \lab{Sz_rec} \ee 
where
\be
\Phi_n^*(z) = z^n \olsi{\Phi}_n(z^{-1}).
\lab{Phi*} \ee
The Verblunsky parameters $a_n$ satisfy the condition
\be
|a_n| <1 . \lab{a<1} \ee
In what follows we shall only consider polynomials $\Phi_n(z)$  with real coefficients $a_n$. These $\Phi_n(z)$ are then orthogonal on the unit circle
\be
\int_{-\pi}^{\pi} \Phi_n(\exp(i \theta))  \Phi_m(\exp(-i \theta)) \rho(\theta) d \theta = h_n \delta_{nm}
\lab{ort_real} \ee
with $\rho$ a positive and symmetric weight function:
\be
\rho(-\theta) = \rho(\theta).
\lab{rho_sym} \ee
The normalization constants
\be
h_0 =1, \quad h_n = (1-a_0^2)(1-a_1^2) \dots (1-a_{n-1}^2) , \quad  n=1,2,\dots
\lab{h_n} \ee
are all positive  $h_n>0, \: n=0,1,\dots$ because of condition \re{a<1}.

It is useful to introduce the Laurent polynomials
\be
\psi_{2n}(z) = z^n \Phi_{2n}(1/z), \quad \psi_{2n+1}(z) = z^{-n} \Phi_{2n+1}(z)
\lab{psi_def} \ee
that satisfy similar orthogonality relation, namely
\be
\int_{-\pi}^{\pi} \psi_n(\exp(i \theta))  \psi_m(\exp(-i \theta)) \rho(\theta) d \theta = h_n \delta_{nm}.
\lab{ort_psi} \ee
The Laurent polynomials $\psi_n(z)$ prove appropriate to identify the sought after eigenvalue problems in the CMV-formalism \cite{Simon}.
Consider the vector 
\be
{\vec \psi(z)} = \{\psi_0(z), \psi_1(z), \psi_2(z) , \dots \}. 
\lab{vec_psi} \ee 
The following relation holds
\be
{\vec \psi(1/z)} = {\mathcal M}_1 {\vec \psi(z)}
\lab{RM1} \ee
with $\mathcal{M}_1$ the semi-infinite block-diagonal matrix
\be
\mathcal{M}_1 =
 \begin{pmatrix}
1 \\
& a_{1} & 1 &  &    \\
  &1-a_1^2 & -a_{1} &  &   \\
   & &  &               a_{3} & 1  \\
   & & &              1-a_3^2 & -a_{3}  \\
  &  & &   & &           a_{5} & 1  \\
   &  & &   & &         1-a_5^2 & - a_{5}  \\
&   &  & & & & & \ddots  \\
 \end{pmatrix}.
\lab{M1_def} \ee 
Similarly
\be
 {z \vec\psi(1/z)} = {\mathcal M}_2 {\vec \psi(z)}
\lab{RM2} \ee
with ${\mathcal M}_2$ the block-diagonal matrix
\be
 \mathcal{M}_2 =
 \begin{pmatrix}
  a_{0} & 1 &  &    \\
  1-a_0^2 & -a_{0} &  &   \\
   &  &              a_{2} & 1  \\
   &  &              1-a_2^2 & -a_{2}  \\
  &  &    & &          a_{4} & 1  \\
   &  &    & &         1-a_4^2 & - a_{4}  \\
&   &  & & & & \ddots  \\
 \end{pmatrix}.
\lab{M2_def} \ee
The matrices ${\mathcal M}_1$ and ${\mathcal M}_2$ are involution operators, i.e. they satisfy the relations
\be
{\mathcal M}_1^2 = {\mathcal M}_2^2 = {\mathcal I},
\lab{inv_MM} \ee
where ${\mathcal I}$ is the identity operator.
The vector ${\vec \psi}$ thus satisfies the generalized eigenvalue problem 
\be
\mathcal{M}_2 { \vec  \psi(z)} = z \mathcal{M}_1 { \vec  \psi(z)}.
\lab{GEVP_psi} \ee
From relations \re{RM1}, \re{RM2} and \re{inv_MM}, it is further seen that $\vec \psi(z)$ satisfies also the {\it ordinary} eigenvalue problem 
\be 
{\mathcal C} \vec \psi = z \vec \psi
\lab{CMV_rec} \ee
where $\mathcal C = \mathcal{M}_1 \mathcal{M}_2$ is a pentadiagonal (CMV) matrix \cite{Simon}.

Consider in addition the set of Laurent polynomials
\be
P_0 =1, \quad P_n = z^{1-n} \Phi_{2n-1}(z) + z^{n-1} \Phi_{2n-1}(1/z) = \psi_{2n-1}(z) + \psi_{2n-1}\left(1/z \right), \; n=1,2,\dots 
\lab{P_Phi} \ee 
These polynomials are symmetric with respect to the operation $z \to 1/z$. Hence, one can conclude that $P_n$ are monic polynomials of the argument $x(z) =z+1/z$:
\be
P_n = P_n(x(z)) = P_n(z+1/z) = (z+1/z)^n + O\left( \left(z+1/z \right)^{n-1} \right).
\lab{P_x} \ee
The first two polynomials are 
\be
P_0 =1, \; P_1(x(z)) = z+1/z -2 a_0.
\lab{P01} \ee
One can define the ``companion" polynomials $Q_n$ through the formula
\be
Q_n(z) = \frac{z^{-n} \Phi_{2n+1}(z) - z^{n} \Phi_{2n+1}(1/z)}{z-z^{-1}} = \frac{\psi_{2n+1}(z) - \psi_{2n+1}(1/z)}{z-z^{-1}}.
\lab{Q_Sz} \ee
Clearly, the polynomials $Q_n$ are again monic and of degree $n$ in the same  argument $x(z)$.
The polynomials $P_n$ and $Q_n$ were introduced by Szeg\H{o} \cite{Szego}. They are orthogonal on the interval $[-2,2]$
\be
\int_{-2}^{2} P_n(x) P_m(x) w(x) d x = 0, \; \mbox{if} \; m\ne n,
\lab{ort_P} \ee   
and 
\be
\int_{-2}^{2} Q_n(x) Q_m(x) w(x)(4-x^2) d x = 0, \; \mbox{if} \; m\ne n,
\lab{ort_Q} \ee 
where
\be
w(x) = \frac{\rho(\theta)}{\sqrt{4-x^2}}
\lab{w_rho} \ee
and $\rho(\theta)$ is understood to be a function of $x$. Indeed, in light of \re{rho_sym}, $\rho(\theta)$ is symmetric and hence depends only on $x = 2 \cos \theta$.
We shall also need the following useful expressions for the polynomials $P_n, Q_n$ \cite{VZ_JOPUC}
\ba
&&P_n(x(z)) = \psi_{2n}(z) +\left( 1+a_{2n-1}\right) \psi_{2n-1}(z), \nonumber \\
&&\left(z-z^{-1}\right)Q_{n-1}(x(z)) = -\psi_{2n}(z) +\left( 1- a_{2n-1}\right) \psi_{2n-1}(z).
\lab{PQ_psi} \ea

\section{Jacobi OPUC and their eigenvalue problem} \label{sect:3}
\setcounter{equation}{0}
The Jacobi OPUC depend on two real parameters $\alpha,\beta$ and are determined by the following Verblunsky parameters
\be
a_n = - \frac{\alpha + 1/2+ (-1)^{n+1} \left(  \beta +1/2 \right) }{n+\alpha+\beta+2}, \quad n=0,1,2,\dots
\lab{an_sol} \ee
These OPUC are orthogonal on the unit circle
\be
\int_{0}^{2\pi} \rho(\theta) \Phi_n\left(e^{i \theta}\right) \Phi_m \left(e^{-i \theta}\right) d \theta = 0, \; n \ne m,
\lab{ort_JOPUC} \ee
with respect to the positive weight function
\be
\rho(\theta) = \left( 1-\cos \theta \right)^{\alpha+1/2}  \left( 1+ \cos \theta \right)^{\beta+1/2}.
\lab{w_JOPUC} \ee
We have shown in \cite{VZ_JOPUC} that these polynomials are CMV bispectral in that the associated Laurent polynomials $\psi_n(z)$ satisfy the eigenvalue equation
\be
\mathcal{K} \psi_n(z) = \mu_n \psi_n(z)
\lab{eigen_JOPUC} \ee
where $\mathcal{K}$ is the following first order differential operator of Dunkl type
\be
\mathcal{K}= z \partial_z  + G(z) \left(R-\mathcal{I} \right)
\lab{KJop} \ee
with
\be
G(z) = \frac{z\left((\alpha+\beta+1)z +\alpha-\beta \right)}{1-z^2}
\lab{G} \ee
and where $R$ denotes the reflection operator
\be
Rf(z) = f(1/z).
\lab{R_def} \ee
The eigenvalues are
\be
\mu_n = \left\{ -n/2,  \; n \; \mbox{even} \atop  (n+1)/2+\alpha+\beta+1 , \; n \; \mbox{odd} . \right.
\lab{l_JOPUC} \ee
Moreover, as pointed out in \cite{VZ_JOPUC}, the operators $\mathcal{K}$, $\mathcal{M}_1$ and $\mathcal{M}_2$ realize an algebra that has been called the circle Jacobi algebra and whose defining relations are: 
\ba
&&\{\mathcal{K}, M_1\} = (\alpha+\beta+1) \left( M_1-\mathcal{I}\right),  \nonumber \\
&&\{\mathcal {K}, M_2\} = (2+\alpha+\beta) M_2 +(\alpha-\beta)\mathcal{I}.
\lab{KM12} \ea
It plays the role of a ``hidden" symmetry algebra for this family of orthogonal polynomials as the fundamental properties of the Jacobi OPUC (Verblunsky parameters, eigenvalue operator $\mathcal{K}$, etc) can be obtained from its representations.

\section{Sieved Jacobi OPUC} \label{sect:4}
\setcounter{equation}{0}
This section focuses on the sieved Jacobi OPUC that we present next and whose eigenvalue equation we proceed to identify.

Let $a_n, n=0,1, \dots$ be the set of Verblunsky parameters associated to the Jacobi OPUC $\Phi_n(z)$. Let
$N$ be a fixed natural number $N = 1,2,\dots$  The sieved Jacobi OPUC $\Phi_n(z;N)$ are determined by the
Verblunsky parameters $a_n(N)$ that are given in terms of the parameters $a_n$ through the relations \cite{Ismail_Li}:
\be
a_n(N) = \left\{ a_{k-1}  \quad \mbox{if} \; n=Nk-1  \atop  0 \quad \quad \mbox{oterwise} .\right.
\lab{a_n(N)} \ee
It has been shown in \cite{Ismail_Li} that
\be
\Phi_n(z;N) = z^j \Phi_k(z^N), \quad \mbox{where} \quad n=Nk+j, \quad j=0,1,\dots, N-1.
\lab{Phi(N)} \ee 
Notice that the ``non-sieved" Jacobi OPUC $\Phi_n(z)$ correspond to the case $N=1$:
\be
\Phi_n(z) = \Phi_n(z;1).
\lab{Phi_nons} \ee
We shall make use of the corresponding Laurent CMV-polynomials 
\be
\psi_{2n}(z;N) = z^n \Phi_{2n}\left(z^{-1};N \right), \;  \psi_{2n+1}(z;N) = z^{-n} \Phi_{2n+1}\left(z;N \right).
\lab{sieved_psi} \ee
One could expect relations between $\psi_n(z;N)$ and $\psi_n(z) =\psi_n(z;1)$ that are similar to formula  \re{Phi(N)}. They indeed exist, but depend on the parity of the numbers $N,n,j$.\\
If $n$ is {\it even}, one has 
\be
\psi_n(z;N) = \left\{  z^{-j/2} \psi_k \left(z^N \right),  \;  j \; \mbox{even}   \atop   z^{(j+1)/2} \psi_k \left(z^{-N} \right), \;  j \; \mbox{odd}.\right.
\lab{psi-psi_n_even} \ee
If $n$ is {\it odd} and $N$ is {\it even}, one finds
\be
\psi_n(z;N) = \left\{  z^{(N-j)/2} \psi_k \left(z^{-N} \right),  \;  j \; \mbox{even}   \atop   z^{(-N+j+1)/2} \psi_k \left(z^{N} \right), \;  j \; \mbox{odd}.\right.
\lab{psi-psi_n_odd_N_even} \ee
Finally, if $n$ is {\it odd} and $N$ is {\it odd}, one gets
\be
\psi_n(z;N) = \left\{  z^{(-N+j+1)/2} \psi_k \left(z^{N} \right),  \;  j \; \mbox{even}   \atop   z^{(N-j)/2} \psi_k \left(z^{-N} \right), \;  j \; \mbox{odd}.\right. ,
\lab{psi-psi_n_odd_N_odd} 
\ee
where
\begin{equation}
    n=Nk+j, \quad j=0,1,\dots, N-1.
\end{equation}

We already know that the operators $R$ and $zR$ act in block-diagonal fashion in the basis  $\psi_n(z)$. The same property is true in the case of the sieved OPUC for the reflection operators $R_j$ and $zR_j$, with
\be
R_j f(z)= f\left( q^j/z\right), \; j=0,1,2,\dots, N-1,
\lab{R_i_def} \ee
and where $q$ is a primitive $N$-th root of unity. In particular, one can take
\be
q=\exp\left(\frac{2 \pi i}{N} \right).
\lab{q_def} \ee
To be precise, using the vector notation of \eqref{vec_psi}, one has
\be
\vec {\psi} \left(q^j/z;N \right) =\mathcal{M}_{1;j} \vec{\psi}(z;N), \quad z \vec{\psi} \left(q^j/z;N \right) =\mathcal{M}_{2;j} \vec{\psi}(z;N), \quad j=0,1,2,\dots, N-1,
\lab{M_i} \ee 
 where  $\mathcal{M}_{1;j}$ and $\mathcal{M}_{2;j}$ are block-diagonal matrices having structures similar respectively to $\mathcal{M}_1$ and $\mathcal{M}_2$ (see \re{M1_def} and \re{M2_def}). That \eqref{M_i} holds is seen as follows. When $j=0$, these relations amount to formulas \re{RM1}-\re{M2_def} which are true for any OPUC. The validity of \eqref{M_i} for $j=1,2, \dots, N-1$, i.e. in the presence of the factors $q^j$, is then confirmed like this. Since $q^{-jN}=1$, one observes from the identities \eqref{psi-psi_n_even}-\eqref{psi-psi_n_odd_N_odd} that
\be
\psi_n(q^{-j}z ;N) = \omega_j(N) \psi_n(z;N),  \quad |\omega_j(N)|=1
\lab{Phi_omega} \ee 
where $\omega_j(N)$ is a phase that depends on $N$ and $j$ according to their parities. Hence upon the replacement  $z \to q^{-j}z$ in the equalities \re{M_i} for $j=0$, we see that these formulas preserve their forms up to rescalings of the entries of $\mathcal{M}_{1;0}$ and $\mathcal{M}_{2;0}$ that yield the matrices $\mathcal{M}_{1;j}$ and $\mathcal{M}_{2;j}$. These $\mathcal{M}_{1;j}$ and $\mathcal{M}_{2;j}$ are moreover reflection matrices as they express the actions of the reflection operators $R_j$.

The ``sieved" polynomials $P_n(z;N)\:$ and $ \: Q_n(z;N)$ that are orthogonal  on the real line will be considered attentively in section 6. We record here their definition in terms of the Laurent CMV-polynomials $\psi_{2n}(z;N)$ and $\psi_{2n+1}(z;N)$:
\ba
&&P_n(x\left(z);N \right) = \psi_{2n}(z;N) + \left(1+a_{2n-1}(N)\right) \psi_{2n-1}(z;N) \\
&& \left(z-z^{-1} \right)Q_{n-1}(x\left(z);N \right) = -\psi_{2n}(z;N) + \left(1-a_{2n-1}(N)\right) \psi_{2n-1}(z;N).
\lab{PQ_sieved} \ea 
The sieved Jacobi OPUC  $\Phi_n(z;N)$ and the corresponding PRL $P_n (x(z);N)$ and $Q_n (x(z);N)$ have been independently  presented by Askey \cite{Askey} and  Badkov \cite{Badkov}. See also \cite{Charris} for a detailed analysis of the properties of these polynomials. The polynomials $P_n (x(z);N)$ and $Q_n (x(z);N)$ are respectively called the sieved
Jacobi polynomials of the first and second kind (see, e.g. \cite{Ismail_Li}). Like the Jacobi PRL, in view of \eqref{an_sol}, \eqref{a_n(N)}, \eqref{sieved_psi} and \eqref{PQ_sieved}, the polynomials  $P_n (x(z);N)$ and $Q_n (x(z);N)$ depend on the two parameters $\alpha$ and $\beta$. An important special case arises when these parameters are taken to be equal, that is when $\alpha = \beta$; this leads to the sieved ultraspherical polynomials on the real line that were first introduced in \cite{AAA}. See Section 7.

To the best of our knowledge, the bispectrality of the sieved Jacobi polynomials has not yet been established for either their OPUC or PRL versions. One expects that these OPs satisfy an eigenvalue equation with respect to an operator acting on the argument $z$. This belief is based on the fact that the sieved Jacobi polynomials can be obtained from a family of the Askey polynomial scheme through a non-trivial limit as $q$ goes to a root of unity. Since the Askey-Wilson polynomials are bispectral, one would assume the same to be true for the sieved Jacobi OPUC in particular. However, taking the straightforward limit as $q$ goes to a root of unity of the $q$-difference equation of (special) Askey-Wilson polynomials does not give the desired result. We here propose an alternate approach to show that the sieved Jacobi OPUC are bispectral. It will rely on the observation made recently \cite{VZ_JOPUC} that the ordinary Jacobi OPUC (the case $N=1$), are CMV-bispectral with their associated Laurent polynomials being eigenfunctions of a differential operator of Dunkl type that acts on the variable $z$. However, in the sieved case, one will find not surprisingly that the Dunkl operator will not only involve the reflection operator $R$, but also {\it the cyclic reflections}  $R_i$ introduced in \re{R_i_def}. This is in line with the fact that all the operators $R_i$ were seen to intervene in the construction of the recurrence relations of CMV type: $\mathcal{M}_{1;j}\mathcal{M}_{2;j}\vec{\psi}(z;N)=z\vec{\psi}(z;N)$, for $j=0,\dots, N-1$. The existence of this set of equations is a specific property of the sieved OPUC that is of course absent for the ordinary OPUC.
We are now ready to state one of our main results:
\begin{pr}
The CMV-Laurent polynomials $\psi_n(z;N)$ associated to the sieved Jacobi OPUC $\Phi_n(z;N)$ satisfy the eigenvalue equation
\be
L(N) \psi_n(z;N) = \lambda_n(N) \psi_n(z;N), 
\lab{eig_SOPUC} \ee
where the operator $L(N)$ is
\be
L(N) = z\partial_z + \sum_{k=0}^{N-1} A_k(z;N) \left( R_k - \mathcal{I}\right) 
\lab{L(N)} \ee
with
\be
A_k(z;N)= \sigma_k \frac{z^2}{q^k- z^2}, \; \sigma_k = \alpha+\beta+1  +(-1)^k (\alpha-\beta)
\lab{A_ev} \ee
for even $N$ and
\be
A_k(z;N) = \frac{(\alpha+\beta+1)z^2 + \rho_k(\alpha-\beta)z}{q^k- z^2}, \; \rho_k = \left\{  q^{k/2}, \; k \: \mbox{even}  \atop q^{(k-N)/2},  \; k \: \mbox{odd} \right. 
\lab{A_odd} \ee
for $N$ odd.
\\

\noindent The eigenvalues are
\be
\lambda_n(N) = \left\{ -n/2, \; n \: \mbox{even}  \atop (n+1)/2 + (\alpha+\beta+1)N, \; n \: \mbox{odd}.   \right.
\lab{lambda(N)} \ee
\end{pr}
\begin{proof}
    
This proposition is rooted in the eigenvalue equation \re{eigen_JOPUC} that the Jacobi OPUC obey. The proof proceeds through a direct verification of \eqref{eig_SOPUC} as defined by \eqref{L(N)}-\eqref{lambda(N)}. In broad strokes, it goes as follows. One uses the relations between the sieved CMV Laurent polynomials $\psi_n(z;N)$ and the ordinary ones $\psi_k(z)$ provided by equations \eqref{psi-psi_n_even}-\eqref{psi-psi_n_odd_N_odd} to write the former in terms of the latter in the following fashion
\be
\psi_n(z;N) = z^{\nu} \psi_k\left(z^N \right) \quad \mbox{or} \quad  \psi_n(z;N) = z^{\nu} \psi_k\left(z^{-N} \right),
\lab{psi_nu} \ee
where $\nu$ is an integer that depends on $n$, $N$ and $j$ with $n=Nk+j$. 
With $\psi_k'(z)$ the derivative of $\psi_k(z)$, we have from $\mathcal{K} \psi_k(z) = \mu_k \psi_k(z)$, \re{eigen_JOPUC}, that
\be
z^N \psi_k'\left( z^N\right) =   \left( G\left( z^N \right) + \mu_k \right)  \psi_k\left( z^N\right)  - G\left( z^N\right) \psi_k\left( z^{-N}\right)
\lab{der_psiN} \ee
and a similar formula with $z^N$ replaced with $z^{-N}$. Substituting $\psi_n(z;N)$ as given by its expression of the type \eqref{psi_nu} into \re{eig_SOPUC} and 
replacing the derivative $\psi_k'\left( z^N\right)$ by the relation \eqref{der_psiN} or as the case maybe, replacing $\psi_k'\left( z^{-N} \right)$ using the the relation with $z^N \rightarrow z^{-N}$, one will see that \eqref{eig_SOPUC} is transformed into an equation of the form
\be
F^{(1)}_n(z) \psi_n\left( z^N\right) + F^{(2)}_n(z) \psi_n\left( z^{-N}\right) =0
\lab{FF} \ee
with the functions $F^{(1)}_n(z)$  and $F^{(2)}_n(z)$ containing $A(z^N), \mu_k$ and $\lambda_n(N)$. This is what needs to be proved at that point to confirm that Proposition 1 is true. As it turns out, it is possible to show that these functions $F^{(1)}_n(z)$  and $F^{(2)}_n(z)$ are identically zero separately to complete the proof.

As the reader will have appreciated, there are many cases to check in that way however. First $L(N)$ takes a different form depending on the parity of $N$. Second, the relations between the sieved CMV Laurent polynomials and the non-sieved ones split in numerous cases as per \eqref{psi-psi_n_even}-\eqref{psi-psi_n_odd_N_odd}. While the results have been validated in all these situations, it would be tedious to go through all of them here. However, for the sake of giving some details of the proof, we shall provide some indications on the computations in only one case namely the one with $n$, $N$, (and $j$ then) all even. With these conditions on the parameters, from 
\eqref{psi-psi_n_even} we have
\begin{equation}
    \psi_n(z;N)=z^{-\frac{j}{2}}\psi_k(z^N), \quad j=0, 2,\dots, N-2. \label{psi_special}
\end{equation}
From the Dunkl eigenvalue equation \eqref{eigen_JOPUC}-\eqref{l_JOPUC} for the CMV Laurent polynomials $\psi_k(z)$ we have:
\begin{equation}
    z^N\psi_k'(z^N) =- \big{[}\frac{(\alpha+\beta+1)z^{2N}+(\alpha-\beta)z^N}{1-z^{2N}}\big{]}\big{(}\psi_k(z^{-N})-\psi_k(z^N)\big{)}  -\frac{k}{2}\psi_k(z^N).\label{psiprime}
\end{equation}
Inserting \eqref{psi_special} in the proposed eigenvalue equation \eqref{eig_SOPUC} for the sieved Jacobi CMV OPUC, we have:
\begin{equation}
    (\frac{n}{2} - \frac{j}{2})z^{-\frac{j}{2}} \psi(Z^N) + Nz^{-\frac{j}{2}} Z^N\psi_k'(Z^N) + \sum_{l=0}^{N-1} A_l(z;N)(R_l-I)\big{(}z^{-\frac{j}{2}}\psi_k(z^N)\big{)}=0 \label{eig_int}
\end{equation}
with 
\begin{equation}
    A_l(z;N) = \big{[}(\alpha+\beta+1)+(-1^l)(\alpha-\beta)\big{]}, \quad l=0,\dots, N-1,
\end{equation}
for the particular case we are considering. Now
\begin{equation}
    (R_l-I) \big{(}z^{-\frac{j}{2}}\psi_k(z^N)\big{)}= q^{-\frac{lj}{2}}z^{\frac{j}{2}}\psi_k(z^{-N}) - z^{-\frac{j}{2}}\psi_k(z^N).
\end{equation}
We observe that sums of the type $\sum_{l=0}^{N-1} \frac{q^{lh}}{q^l-z}$ occur in the last term of \eqref{eig_int}. Let us focus on these for a moment and consider to that end the contour integral 
\begin{equation}
    \oint_{|w|>1} dw \frac{w^h}{(w^N-1)(w-z)}, \qquad h\in \mathbb{N},
\end{equation}
over a circle of radius larger than $1$ in the complex plane. Assume $h$ is a non-negative integer smaller than $N$. The integrand has poles on the unit circle at the roots of unity $w_l=q^l$ with $q=e^{\frac{2i\pi}{N}}$ and at $w=z$. Since it is analytic outside the unit circle, by Cauchy's theorem, the sum of the residues will be zero. Keeping in mind that 
\begin{equation}
    \lim_{w\rightarrow q^l}\;\frac{(w-q^l)}{\Pi_{k=0}^{N-1}(w-q^k)}=\frac{N}{q^l}, 
\end{equation}
we thus have:
\begin{equation}
    \sum_{l=0}^{N-1}\frac{q^{l(h+1)}}{q^l-z}=N\frac{z^h}{1-z^N}. \label{sumgen}
    \end{equation}
Applying this formula while recalling that $h$ must be positive, that $q^N=1$ and respecting the range of $j$, we get for the needed sums:
 \begin{align}
   &z^2\sum_{l=0}^{N-1} \frac{q^{-\frac{lj}{2}}}{q^l-z^2} = N\frac{z^{2N-j}}{1-z^{2N}},  \label{sum1}   \\
     & z^2\sum_{l=0}^{N-1} \frac{(-1)^l q^{-\frac{lj}{2}}}{q^l-z^2} 
      =z^2\sum_{l=0}^{N-1} \frac{q^{(\frac{N}{2}-\frac{j}{2})l}}{q^l-z^2} 
    = N\frac{z^{N-j}}{1-z^{2N}}.  \label{sum2}
\end{align}

Putting all the pieces together in \eqref{eig_int}, i.e. replacing $z^N\psi_k'(z^N)$ by the r.h.s. of \eqref{psiprime} and using the sums \eqref{sum1} and \eqref{sum2}, one readily sees that the factors of $\psi_k(z^N)$ and of $\psi_k(z^{-N})$ vanish if one recalls that that $n=Nk+j$. All the other possible situations for $n$, $N$ and $j$ can be treated analogously to confirm that Proposition 1 holds.
\end{proof}

The operator $L(N)$ can be presented in the equivalent form using again the sums \eqref{sumgen}:
\be
L(N) = z\partial_z + \sum_{k=0}^{N-1} A_k(z;N)R_k +B(z) \mathcal{I},
\lab{L(N)_B} \ee
where 
\be
B(z) = -\sum_{k=0}^{N-1}A_k(z;N) = N \frac{(\alpha+\beta+1)z^{2N} + (\alpha-\beta)z^N}{z^{2N}-1}.
\lab{B_expl} \ee
Note that in contrast to the expressions of the coefficients $A_k(z;N)$, the one of $B(z)$ does not depend on the parity of $N$.
We also have the following result regarding the self-adjointness of $L(N)$:
\begin{pr} \lab{Prop2}
The operator $L(N)$ is self adjoint on the unit circle with respect to the scalar product
\be
(f(z),g(z)) = \int_0^{2 \pi} f\left( e^{i \theta} \right) \bar g \left(e^{-i\theta}  \right) \rho(\theta;N) d \theta,
\lab{scal_N} \ee
where 
\be
\rho(\theta;N) = \left( 1-\cos N\theta \right)^{\alpha+1/2}  \left( 1+ \cos N\theta \right)^{\beta+1/2}. 
\lab{rho_N} \ee
In other words, this amounts to the following operator relation
\be
L(N)^{\dagger} \rho(\theta;N) = \rho(\theta;N) L(N),
\lab{self_L} \ee
with $L(N)^{\dagger}$ the adjoint of $L(N)$ and where it is assumed that $z=e^{i \theta}$. 
\end{pr} 
\begin{proof}
  Since $z\partial_z=-i\partial_\theta$, given \eqref{L(N)_B}, the adjoint of $L(N)$ reads:
  \begin{equation}
    L(N)^{\dagger}=-i\partial_\theta + \sum_{k=0}^{N-1} A_k(z,N)R_k+B^*(z) \label{Ldagger}
  \end{equation}
  taking into account that $R_k A_k(z)^*=A_k(z)R_k$ as is directly seen from the definition \eqref{R_i_def}, \eqref{q_def} of $R_k$ and the expressions \eqref{A_ev} and \eqref{A_odd} of $A_k(z)$. Now note that $\rho(\theta;N)=\rho(N\theta)$ with $\rho(\theta)$ the weight \eqref{w_JOPUC} of the Jacobi OPUC and remark that the action of $R_k$ translates as follows on functions of $\theta$:
  \begin{equation}
      R_kf(\theta)=f(-\theta+\frac{2\pi k}{N}).
  \end{equation}
  To prove \eqref{self_L}, we first observe that $R_k\rho(\theta;N)=\rho(\theta;N)R_k$, namely that $\rho(\theta;N)$ is invariant under the inversions $R_k$. Indeed, using the periodicity of $\rho(\theta)$ and its reflection invariance \eqref{rho_sym}, we have:
  \begin{equation}
     R_k\rho(\theta;N)=\rho(-N\theta+2\pi k)R_k=\rho(-N\theta)R_k=\rho(N\theta)R_k= \rho(\theta;N)R_k.
  \end{equation}
Therefore showing that $L(N)^{\dagger} \rho(\theta;N) = \rho(\theta;N) L(N)$ reduces to checking that 
\begin{equation}
    -i\partial_\theta \;\ln \left(\rho(\theta;N)\right)=B(z)-B^*(z).
\end{equation}
One has,
\begin{equation}
    -i\partial_\theta \;\ln \left(\rho(\theta;N)\right) = -\frac{iN}{\sin N\theta}[(\alpha+\beta+1) \cos N\theta + (\alpha-\beta)]. \label{partial_ln}
\end{equation}
Writing \eqref{B_expl} in the form:
\begin{equation}
    B(z)=\frac{N}{z^N-\bar{z}^N}[(\alpha+\beta+1)z^N + (\alpha-\beta)], \quad \text{with} \quad \bar{z}=z^*=z^{-1},
\end{equation}
given that $\alpha$ and $\beta$ are real parameters, we readily find that
\begin{equation}
    B(z)-B^*(z)=\frac{N}{z^N-\bar{z}^N}[(\alpha+\beta+1)(z^N+\bar{z}^N)+ 2(\alpha-\beta)]
\end{equation}
which coincides when $z=e^{i\theta}$ with the expression \eqref{partial_ln} found for $-i\partial_\theta \;\ln \left(\rho(\theta;N) \right)$.
\end{proof}

For the ultraspherical case, i.e. when $\beta=\alpha$ the expressions for $A_k(z;N)$ and $B(z)$ become much simpler:
\be
A_k(z;N) = \frac{(2 \alpha+1)z^2}{q^k-z^2}
\lab{A_ultra} \ee
and 
\be
B(z) = N \frac{(2\alpha+1)z^{2N} }{z^{2N}-1}.
\lab{B_ultra} \ee

\section{Algebraic relations} \label{sect:5}
\setcounter{equation}{0}
When $N=1$ the sieved Jacobi OPUC become the ordinary Jacobi OPUC. There are three basic operators in this case: two reflections $R$ and $zR$ and the operator $L=L(1)$ of which the Jacobi OPUC are eigenfunctions. It was demonstrated in \cite{VZ_JOPUC} that these operators realize the defining  relations of the {\it circle Jacobi algebra} \re{KM12}. It was moreover shown in the same paper that the Verblunsky parameters $a_n$ of the Jacobi OPUC can be derived solely from the representations of this circle Jacobi algebra.

In the case of the sieved Jacobi OPUC with arbitrary $N$, we have as many reflection operators: $R_0,R_1, \dots, R_{N-1}$.
These operators and their products form the dihedral group $D_N$ which has $2N$ elements consisting of the $N$ reflections $R_0, \dots, R_{N-1}$ and the $N$ rotations $T_0, T_1, \dots, T_{N-1}$ defined as
\be
T_k f(z) = f\left( q^k z \right). \label{rot}
\ee 
The rotations arise through the multiplication of two reflections
\be
R_k R_j = T_{j-k}.
\lab{RRT} \ee
The group products are:
\be 
T_k T_j = T_j T_k = T_{k+j}, \quad T_k R_j = R_{j-k}, \quad R_j T_k = R_{j+k}.
\lab{rel_D} \ee
The order of the reflections is 2, the order of the rotations is $N$:
\be
R_j^2 =T_j^N = \mathcal{I}.
\lab{RT_order} \ee 
Moreover, it is well known that the two reflections $R_0, R_1$ suffice to generate the dihedral group $D_{N}$.
One can also introduce the $N$ operators $M_j= z R_j, \: j=0,1,\dots, N-1$. They satisfy the relations
\be
M_j M_k = q^j T_{k-j}, \quad M_j R_k = z T_{k-j}, \quad R_k M_j = q^j z^{-1} T_{j-k}.
\lab{zR2} \ee  

Consider now the commutation or anticommutation relations of the operators $L(N), R_k, T_j$. First, one notices that all rotations commute with the operator $L(N)$ that has the sieved Jacobi OPUC as eigenfunctions:
\be
[L(N), T_j] =0.
\lab{comm_LT} \ee  
The relations \re{comm_LT} follow immediately from the observation in view of the generic form of $\psi_n(z;N)$ given in \eqref{psi_nu} that
\be
T_j \psi_n(z;N) = \psi_n\left( q^j z;N \right) = \omega_{n,j}  \psi_n(z;N),
\lab{T_psi} \ee
where $\omega_{n,j}$ are some coefficients such that $|\omega_{n,j}|=1$.
The commutation relations between $L(N), R_j$ and $M_j$ are more involved and depend on the parity of $N$. 
For $N$ even we have 
\be
\{R_j,  L(N) \} = N(\alpha+\beta+1) R_j - (2\alpha+1) \sum_{k=0}^{N/2-1} T_{2k+j}  -(2\beta+1) \sum_{k=0}^{N/2-1} T_{2k+j+1}, \quad j=0,1,\dots, N-1 
\lab{acm_even_1} \ee
and
\be
\{M_j,  L(N) \} = \left((\alpha+ \beta+1)N +1 \right) M_j, \quad j=0,1,\dots, N-1.
\lab{acm_even_2} \ee
For $N$ odd we have
\be
\{R_j,  L(N) \} = N(\alpha+\beta+1) R_j - (\alpha+\beta+1) \sum_{k=0}^{N-1} T_k, \quad j=0,1,\dots, N-1 
\lab{acm_odd_1} \ee
and
\be
\{M_j,  L(N) \} = \left( N\alpha+ N\beta+ N+1 \right)  M_j + (\alpha- \beta) \sum_{k=0}^{N-1} q^{j+kJ} T_{j+k}, \quad j=0,1,\dots, N-1 ,
\lab{acm_odd_2} \ee
where $J=(N-1)/2$.
These relations can be taken to define a generalization of the circle Jacobi algebra \re{KM12}. Deriving the explicit expressions of the Verblunsky coefficients of the sieved Jacobi OPUC from the construction of the representations of this generalized circle Jacobi algebra alone, is an interesting problem which we defer to future work.


\section{Eigenvalue equations for the sieved Jacobi polynomials on the real line} \label{sect:6}
\setcounter{equation}{0}
In this section we present the eigenvalue equations for the sieved Jacobi polynomials of the first and second kind on the real line. They will be obtained with the help of the eigenvalue equation \re{eig_SOPUC} of the sieved Jacobi OPUC.  

Introduce the operator
\be
H(N) = L^2(N) -N(\alpha+\beta+1)L(N).
\lab{H(N)} \ee
We have
\be
H(N) \psi_n(z;N) = \t \lambda_n(N) \psi_n(z;N)
\lab{H_psi} \ee
where
\be
\t \lambda_n(N) = \lambda^2_n(N) -N(\alpha+\beta+1)\lambda_n(N).
\lab{tlam} \ee
These eigenvalues satisfy the property
\be
\t \lambda_{2n} = \t \lambda_{2n-1} = \Lambda_n(N) 
\lab{tl_eo} \ee
where
\be
\Lambda_n(N) = n(n+N(\alpha+\beta+1)). 
\lab{Lambda_n} \ee
Apply the operator $H(N)$ to the polynomial $P_n(x(z);N)$. Using formulas \re{PQ_sieved} and \re{tl_eo} we have:
\begin{align}
  H(N) P(x(z);N) &= H(N) \psi_{2n}(z;N) + \left(1+a_{2n-1}(N)\right) H(N) \psi_{2n-1}(z;N) \nonumber\\  
  &= \Lambda_n(N) P_n(x(z);N).
\lab{HPL}
\end{align}
Hence the polynomials $P_n(x(z);N)$ are eigenfunctions of the operator $H(N)$.

\noindent Similarly we have 
\be
\t H(N) Q_n(x(z);N)= \Lambda_{n+1}(N) Q_n(x(z);N),
\lab{HQL} 
\ee  
where the operator $\t H(N)$ is related to the operator $H(N)$ by the similarity transformation
\be
\t H(N) = \phi^{-1}(z) H(N) \phi(z)
\lab{tH} \ee
with 
\be
\phi(z) = z-z^{-1}.
\lab{phi(z)} \ee
We thus arrive at
\begin{pr}
The sieved Jacobi polynomials $P_n(x(z);N)$ of the first kind satisfy the eigenvalue equation\re{HPL} while the sieved Jacobi polynomials of the second kind $Q_n(x(z);N)$ verify the eigenvalue equation \re{HQL}, where the operators $H(N)$ and $\t H(N)$ are defined by \re{H(N)} and \re{tH}.  
\end{pr}
In light of definition \re{H(N)}, we can provide an explicit expression for the operator $H$. Indeed, from formula \re{L(N)_B} for the operator $L(N)$ we have 
\be
H(N) = z^2 \partial_z^2 + C(z;N) \partial_z + \sum_{k=0}^{N-1} D_k(z;N)\left(  R_k - \mathcal{I} \right) + \sum_{k=1}^{N-1} E_k(z;N) T_k
\lab{H(N)_sum} 
\ee
where the operators $T_k$ are defined as in \eqref{rot}.
The coefficients $D_k(z;N)$ and $E_k(z;N)$ are given by
\be
D_k(z;N) = A_k(z;N) \left( B(z) + B\left( q^k/z \right)  -N(\alpha+\beta+1) \right) + z A'_k(z;N),
\lab{E_k_gen} \ee 
\be
E_k(z;N) =  \sum_{i=0}^{N-1} A_i(z;N) A_{i+k}\left( q^i/z;N\right).
\lab{G_k_sum} \ee
In the sum \re{G_k_sum}, one can take into account the cyclic property of the coefficients $A_k(z)$:
\be
A_{k+N}(z;N) = A_k(z;N)
\lab{cycl_A} \ee
which follows straightforwardly from the explicit expressions \re{A_ev} and \re{A_odd}. 
The calculations then show that all the factors of the shift operators $T_k$ vanish:
\be
E_k(z) = 0, \quad k=1,\dots, N-1.
\lab{G=0} \ee
Moreover, it is seen that 
\be
 B(z) + B\left( q^k/z \right)  -N(\alpha+\beta+1) =0, \quad k=0,1, \dots, N-1
 \lab{BB=0} \ee
and that hence the coefficients in front of the reflection operators $R_k$ simplify to
\be  
D_k(z) = z A'_k(z;N).
\lab{E=zA} \ee
Finally, the term in front of the first derivative is found to be
\be
C(z) = z \left(1+N \left(\alpha +\beta +1\right)+\frac{2 N \left(\alpha +\beta +1+\left(\alpha -\beta \right) z^{N}\right)}{z^{2 N}-1}\right).
\lab{C(z)} \ee
Note last that on the space of the symmetric polynomials $P_n(z), Q_n(z)$ which only depend on the argument $x=z+1/z$, the reflection operators $R_k$ can be replaced by the rotation operators $T_{-k}$. Indeed, for any symmetric Laurent polynomial $f(z)=F(z+1/z)$, we have
\be
R_k f (z) = F\left(q^k/z + z q^{-k} \right) = f\left(zq^{-k}\right) = T_{-k} f(z).
\lab{R=T} \ee
 \noindent This is summed up as follows.
\begin{pr}\lab{Prop4}
The eigenvalue operator $H(N)$ for sieved Jacobi polynomials $P_n(z;N)$ has two equivalent expressions: 

    \begin{numcases}{H(N)=}
z^2 \partial_z^2 + C(z) \partial_z + z \sum_{k=0}^{N-1} A_k'(z;N) \left(R_{k} - \mathcal{I}  \right) \lab{H_expl1} \\
z^2 \partial_z^2 + C(z) \partial_z + z \sum_{k=1}^{N-1} A_k'(z;N) \left(T_{-k} - \mathcal{I}  \right) \lab{H_expl2} 
    \end{numcases}       
where the formulas for $A_k(z;N)$ and $C(z)$ are given by \re{A_ev}-\re{A_odd} and \re{C(z)} and where $R_k$ and $T_k$ are the reflection and  rotation operators
\be
R_k f(z) = f\left(q^k/z\right), \quad   T_k f(z) = f\left( z q^k\right)
\lab{T_k_f}. \ee
\end{pr}

\begin{remark}
The two operators \eqref{H_expl1} and \eqref{H_expl2} providing expressions for $H(N)$ in Proposition \ref{Prop4} have the same action on the space of symmetric Laurent polynomials i.e. such that $f(z)=f(1/z)$ and are hence indistinguishable when acting on symmetric polynomials like $P_n(z)$ or $Q_n(z)$. Hovever, these operators are not equivalent on the space of arbitrary Laurent polynomials. 
\end{remark}
\begin{remark}
   It easily seen that the operator $H(N)$ is self-adjoint as it is obtained, according to formula \re{H(N)}, from the operator $L(N)$ that was shown to be self-adjoint in Proposition \ref{Prop2}. 
\end{remark}

\section{Commuting operators} \label{sect:7}
\setcounter{equation}{0}
We shall observe in this section that the sieved Jacobi polynomials $P_n(z)$ are also eigenfunctions of the discrete operators
\begin{equation}
   Y_m = T_m + T_{-m}
\lab{Y_k_P}  
\end{equation}
built from the rotation operators $T_{\pm{m}}$ defined in \eqref{rot}. It will furthermore be pointed out at the end of this section that the sieved polynomials of second kind $Q_n(z)$ possess a similar property.
\begin{pr} \label{prop:5}
    The sieved Jacobi polynomials $P_n(N)$ are eigenfunctions of the operators $Y_m$:
    \begin{equation}
        Y_m P_n(z) = \omega_{m,n} P_n(z), \quad m=1,2,\dots, N-1. \lab{YP}
    \end{equation}
\end{pr}
\begin{proof}
    This follows from the relations between $\psi_n(z;N)$ and $\psi_n(z)=\psi_n(z;1)$ recorded in Section \ref{sect:4}. Indeed, assume that 
$2n-1=kN +j, \: j=0,1,\dots, N-1$.
We have from \re{P_Phi}
\begin{align}
 Y_m P_n(z)&= P_n(zq^m) + P_n(zq^{-m}) \nonumber  \\
&=\psi_{2n-1}(zq^m;N)+ \psi_{2n-1}(zq^{-m};N) + \psi_{2n-1}(z^{-1} q^m;N) + \psi_{2n-1}(z^{-1} q^{-m};N).
\lab{YP_psi}    
\end{align}
With the help of formulas \re{psi-psi_n_odd_N_even} and \re{psi-psi_n_odd_N_odd}, we conclude that \re{YP} is verified with
\be
\omega_{m,n} = q^{\frac{m(N-j)}{2}} + q^{-\frac{-m(N-j)}{2}}
\lab{om_1} \ee
if $N$ and $j$ have the same parity or
\be
\omega_{m,n} = q^{\frac{m(-N+j+1)}{2}} + q^{-\frac{-m(-N+j+1)}{2}}
\lab{om_2} \ee 
if $N$ and $j$ have different parity. 
It is obvious from \eqref{om_1} and \eqref{om_2}, that $\omega_{m,n}$ can take no more then $N$ different values.
\end{proof}
\begin{remark}
Since according to Proposition \ref{prop:5}, the operators $H$ and $Y_m, m=1, \dots,N-1$,  can be diagonalized simultaneously, we may conclude that they commute among themselves. In view of this, one might replace the eigenvalue equation \eqref{HPL}, $H(N)P_n(z)=\Lambda_n(N)P_n(z)$, by the eigenvalue equation $\mathcal{H}(N)P_n(z)=\Omega_n(N)P_n(z)$ where
\begin{equation}
    \mathcal{H}(N)=H + \sum_{m=1}^{N-1} \tau_m Y_m \label{Hcal}
\end{equation}
with $\tau_i$ arbitrary constants and where consequently $\Omega_n(N)=\Lambda_n(N)+\sum_{m=1}^{N-1} \tau_m \omega_{m,n}$.
We may comment that the operator $H(N)$ is special in that its eigenvalues $\Lambda_n(N)$ are quadratic in $n$ and do not depend on $j$, while this is not so for the eigenvalues $\Omega_n(N)$ of $ \mathcal{H}$ recalling that $j=2n-1 \; \mbox{mod} \; N$.
\end{remark}

The sieved Jacobi polynomials $Q_n(z)$ of the second kind are analogously eigenfunctions of discrete operators $\t Y_m$ which commute among themselves and with the operator 
 $\t H(N)$ that appears in the eigenvalue equation \eqref{HQL}. These $\t Y_m$ are obtained through the similarity transformation that relates $\t H(N)$ to $H(N)$, namely
\be
\t Y_m = \phi^{-1}(z)   Y_m \phi(z), 
\lab{tYY} \ee
where $\phi(z) =z-z^{-1}$.
This yields the explicit expression
\be
\t Y_m = \left(  \frac{  zq^m  -q^{-m} z^{-1}}{z-z^{-1}}\right) T_m +  \left(  \frac{  zq^{-m}  -q^m z^{-1}}{z-z^{-1}}\right) T_{-m}.
\lab{tY} \ee
It follows that the polynomials $Q_n(z)$ are eigenfunctions  of the operator $\t Y_m$ since we have
\be
\t Y_m Q_{n-1}(z) = \omega_{m,n} Q_{n-1}(z),
\lab{YmQ} \ee
with $\omega_{m,n}$ given by \re{om_1} or \re{om_2}.

\section{Special cases} \label{sect:8}
\setcounter{equation}{0}
We shall focus in this section on the interesting families of polynomials that arise in special situations: first, when $N=1$ and $N=2$ and second, when $\alpha=\beta$ for arbitrary $N$.

\subsection{The generalized ultraspherical polynomials}

The case $N=1$ obviously leads to the ordinary (i.e. non-sieved) Jacobi polynomials. Equation \re{HPL} has no terms with reflection operators as is clear from the definition \eqref{H(N)_sum} of $H(1)$ and using \eqref{C(z)}, it is seen to coincide with the standard differential equation for the Jacobi polynomials (presented in terms of the argument $z$ instead of standard $x=z+1/z$). 

Set next $N=2$ . In this case the polynomials $P_n(x)$ satisfy the recurrence relation
\be
P_{n+1}(x) + u_n P_{n-1}(x) = xP_n(x),
\lab{rec_gultra} \ee
where
\be
u_{2n} = \frac{4n(\alpha+n)}{(\alpha+\beta+2n)(\alpha+\beta+2n+1)}, \quad u_{2n+2} = \frac{4(\beta+n+1)(\alpha+n+1)}{(\alpha+\beta+2n+2)(\alpha+\beta+2n+1)}.
\lab{u_gultra} 
\ee
The PRL $P_n(x)$ with coefficients \re{u_gultra} are identified as the generalized ultraspherical (or generalized Gegenbauer) polynomials introduced by Szeg\H{o} (see \cite{Bel} and \cite{Chihara} for details).
When $N=2$, the only nontrivial root of unity is $q=-1$; hence equation \re{HPL} will involve only one reflection operator namely,
 $Rf(z) =f(-z)$. This equation will explicitly read as follows:
\be
H P_n(x(z)) = \lambda_n P(x(z)),
\lab{eq_gu} \ee 
where the operator $H=H(2)$ is given by
\be
H = z^2 \partial_z^2 +C(z) \partial_z  - \frac{(2 \beta+1) z^2}{z^2+1} \left(T -\mathcal{I} \right)
\lab{H_gu} \ee
with
\be
C(z) = z \left( 2 \alpha+2 \beta+3 + \frac{4\left(\alpha+\beta+1 +\left(\alpha-\beta z^2 \right)\right)}{z^4-1} \right)
\lab{C_ug} \ee
and with the operator $T$ changing the sign of $z$:
\be
T f(z) = f(-z).
\lab{T_s} \ee
It can be checked that  equation \re{eq_gu} coincides with the Dunkl type differential equation found in \cite{Ben} for the generalized ultraspherical polynomials.

\subsection{The sieved ultraspherical polynomials}

When $\beta=\alpha$ and $N$ is arbitrary, the polynomials $P_n(x;N)$ and $Q_n(x;N)$ are respectively, the sieved ultraspherical polynomials of the first and second kind found in \cite{AAA}. Consider first those of the first kind. They obey the recurrence relation \re{rec_gultra} with coefficients given by
\be
u_{nN} = \frac{2n}{2\alpha+2n+1}, \quad u_{nN+1} = \frac{4\alpha+2n+2}{2\alpha+2n+1}, \quad n=0,1,2\dots
\lab{u_1nd} 
\ee 
and
\be
u_n=1 \quad \text{for all other} \quad n. \nonumber
\ee
\noindent The next proposition regarding these polynomials stems from the results of Section 6.
\begin{pr}
The sieved ultraspherical polynomials of the first kind $P_n(x;N)$ satisfy the eigenvalue equation
\be 
H(N) P_n\left(z+1/z; N \right) = n\left(n+ \left(2 \alpha+1 \right)N \right) P_n\left(z+1/z ; N \right), \; n=0,1,2,\dots,
\lab{EE_ultra} \ee
where the operator $H(N)$ has equivalently either one of the following two expressions
\begin{numcases}{H(N)=}
   z^2 \partial_z^2 + C(z) \partial_z + \sum_{k=0}^{N-1} B_k(z)\left(R_{k} - \mathcal{I} \right)  
\lab{L_ultra}\\
z^2 \partial_z^2 + C(z) \partial_z + \sum_{k=1}^{N-1} B_k(z)\left(T_{-k} - \mathcal{I} \right) , 
\lab{L_ultra_t}
\end{numcases}
with 
\be
q=\exp\left(\frac{2\pi i}{N} \right), \quad T_k f(z) = \left(q^k z \right)
\lab{q_T} \ee
and
\be
B_k(z) = \frac{2(2 \alpha+1)q^k z^2}{\left( q^k -z^2\right)^2}, \quad C(z) = z \left[1+N(2 \alpha+1) + \frac{2N(2\alpha+1)}{z^{2N}-1} \right]. 
\lab{AC_ultra} 
\ee

\end{pr}

Coming to the sieved ultraspherical polynomials of the second kind, we see that those PRL have the recurrence coefficients 
\be
u_{nN} = \frac{2n}{2\alpha+2n+1}, \quad u_{nN-1} = \frac{4\alpha+2n+2}{2\alpha+2n+1}, \; n=1,2, 3,  \dots
\lab{u_2nd} 
\ee 
and 
\be
u_n=1 \quad \text{for all other} \quad n. \nonumber
\ee
The operator $\t H(N)$ defining the specialization of the eigenvalue equation \eqref{HQL} for the polynomials $Q_n(z;N)$ can be obtained from formula \re{tH}. The spectrum $\Lambda_{n+1}(N)$ is provided by \eqref{Lambda_n} and it will be recalled that we are considering the particular case $\alpha=\beta$. For standardization purposes, the eigenvalue equation will be rewritten in the form $\hat{H}_n(N)Q_n(N)=\Xi_n(N)Q_n(N)$ with $\Xi_0(N)=0$. This will be achieved by subtracting the constant $\Lambda_1(N)$ from $\tilde{H}(N)$. Proceeding with this reformatting, it is readily seen that 
\begin{equation}
    \Xi_n(N)=\Lambda_{n+1}(N)-\Lambda_1(N)=n(n+N(2\alpha+1)+2),
\end{equation}
while in view of \eqref{tH}, $\hat{H}(N)$ is given by
\begin{equation}
    \hat{H}(N)=\left(z-z^{-1} \right)^{-1} H(N) \left(z-z^{-1} \right) - \left(N(2\alpha+1) +1 \right) \mathcal{I}.
\lab{tH_Q} 
\end{equation}
After straightforward computations to perform the conjugation and carry out simplifications, one arrives at the following result.

\begin{pr}
The sieved ultraspherical polynomials of the second kind $Q_n(x;N)$ satisfy the eigenvalue equation
\be 
\hat{H}(N) Q_n\left(z+1/z; N \right) = n\left(n+ \left(2 \alpha+1 \right)N + 2\right) Q_n\left(z+1/z ; N \right), \; n=0,1,2,\dots,
\lab{EQ_ultra} \ee
where $\hat{H}(N)$ takes equivalently one of the following two forms
\begin{numcases}{\hat{H}(N)=}
z^2 \partial_z^2 + \hat{C}(z) \partial_z + \sum_{k=0}^{N-1} \hat{B}_k(z)\left(R_{k} - \mathcal{I} \right)  
\lab{HQ_ultra}\\
z^2 \partial_z^2 + \hat{C}(z) \partial_z + \sum_{k=1}^{N-1} \hat{B}_k(z)\left(T_{-k} - \mathcal{I} \right),  
\lab{HQ_ultra_t}
\end{numcases}
and where 
\be
\hat{B}_k(z) = \frac{q^{2k}-z^2}{q^k\left(z^2-1 \right)} B_k(z), \quad \hat{C}(z) = C(z) +\frac{2z(z^2+1)}{z^2-1}, 
\lab{BC_Q} \ee
with $B_k(z)$ and $C(z)$ given by \re{AC_ultra}. 
\end{pr}

\section{Conclusion} \label{sect:9}
\setcounter{equation}{0}
Summing up, we have found that the CMV Laurent polynomials associated to the sieved Jacobi polynomials on the unit circle are eigenfunctions of a  
Dunkl type differential operator of first order. This paved the way to showing that the sieved Jacobi polynomials of first and second type on the real line satisfy eigenvalue equations involving Dunkl type differential operator of second order.

These results provide a long-awaited demonstration that the sieved Jacobi polynomials are bispectral. In contrast to the case of the classical and even of the little and big -1 Jacobi polynomials \cite{VZ_little-1}, \cite{VZ_big-1} \cite{DVZ}, the operators that intervene have a rather complicated structure that involves many reflections. To the best of our knowledge such operators have not so far appeared in the context of orthogonal polynomials in one variable. 

This study raises various interesting questions. For instance, what are the most general operators of Dunkl type that have orthogonal polynomials as eigenfunctions? Or, can one identify an algebra that would allow a representation theoretic description of the sieved Jacobi polynomials?  These are issues that we plan to explore in the future.

\vspace{23mm}

{\large \bf Acknowledgments.} 
AZ thanks the Centre de Recherches Math\'ematiques (Universit\'e de
Montr\'eal) for its hospitality when this project was initiated.   
The research of LV is supported in part by a research grant from the Natural Sciences and Engineering Research Council
(NSERC) of Canada.

\vspace{5mm}

\bb{99}

\bi{Charris} B.H.Aldana, J.A.Charris and
O.Mora-Valbuena, {\it  On Block Recursions,
Askey's Sieved Jacobi Polynomials
and two related Systems}, Colloqium Mathematicum
{\bf 78} 1998, 57--91.

\bi{AAA}  W.Al-Salam, W. R. Allaway and R.Askey, {\it Sieved ultraspherical polynomials}, Trans. Amer. Math. Soc. {\bf 284} (1984), 39--55.

\bi{Askey} R.Askey, {\it R. Askey, Orthogonal polynomials old and new, and some combinatorial connec-
tions}, in: Enumeration and Design, D. M. Jackson and S. A. Vanstone (eds.), Academic
Press, Toronto, Ont., 1984, 67--84.

\bi{Badkov} V.M.Badkov, {\it Systems of orthogonal polynomials explicitly represented by the Jacobi polynomials}, Mathematical
Notes of the Academy of Sciences of the USSR {\bf 42}, 858--863 (1987).

\bi{Bel} S. Belmehdi, {\it Generalized Gegenbauer polynomials}, J. Comput. Appl Math. {\bf 133} (2001), 195--205.


\bi{BIW}  J. Bustoz, M. E. H. Ismail, and J. Wimp, {\it On sieved orthogonal polynomials VI: differential
equations}, Differential and Integral Equation {\bf 3} (1990) 757--76.


\bi{Ben} Y. Ben Cheikh and M.Gaied, {\it Characterization of the Dunkl-classical symmetric orthogonal polynomials},
Appl. Math. and Comput. {\bf 187}, (2007) 105--114.

\bi{Cher}  I. Cherednik, {\it Double affine Hecke algebras, Knizhnik-Zamolodchikov equations, and Macdonald's operators}, Int. Math. Res. Not. (1992), no. 9, 171--180.

\bi{Chihara}  T. Chihara, An Introduction to Orthogonal Polynomials, Gordon and Breach, NY, 1978.

\bi{DVZ} Derevyagin, M., Vinet, L., Zhedanov, A., {\it CMV matrices and Little and Big -1Jacobi Polynomials}, Constr. Approx. {\bf 36} (2012), 513--535.



\bi{GVA} J. S. Geronimo and W. Van Assche, {\it Orthogonal polynomials on several intervals via a
polynomial mapping}, Trans. Amer. Math. Soc. {\bf 308} (1988), 559--581.

\bi{Ger} Ya. L. Geronimus,\quad {\it Polynomials Orthogonal on a
Circle and their Applications}, \\ Am.Math.Transl.,Ser.1 {\bf
3}(1962), 1-78.

\bi{GV} F.A.Gr\"unbaum and L. Vel\'asquez, {\it The CMV Bispectral Problem}, IMRN, {\bf 2017}, (2017),  5833--5860. Arxiv 1607.01962v2

\bi{Ismail} M.E.H.Ismail, {\it Classical and Quantum orthogonal
polynomials in one variable}. Encyclopedia of Mathematics and its
Applications (No. 98), Cambridge, 2005.

\bi{Ismail_Li} M.E.H.Ismail and X.Li, {\it On sieved orthogonal polynomials. IX: Orthogonality on the unit circle}, Pacific J.Math. {\bf 153} (1992), 289--297.


\bi{KLS} R. Koekoek, P. A. Lesky and R. F. Swarttouw. Hypergeometric orthogonal polynomials
and their q-analogues. Springer Science \& Business Media, 2010.

\bi{Koo} T.Koornwinder, {\it The Relationship between Zhedanov's Algebra $AW(3)$
and the Double Affine Hecke Algebra in the Rank One Case}, SIGMA {\bf 3} (2007), 063.

\bi{KB} T.Koornwinder and F.Bouzeffour, {\it Nonsymmetric Askey-Wilson polynomials as vector-valued polynomials}, Appl.Anal. {\bf 90}, 2011, 731--746. arXiv:1006.1140v3.



\bi{Simon} B.Simon, {\it Orthogonal Polynomials On The Unit
Circle}, AMS, 2005.

\bi{Szego} G. Szeg\H{o}, {\it Orthogonal Polynomials}, American Mathematical Society, 1939. 4th Edition, 1975


\bi{VZ_Askey} L.Vinet and A.Zhedanov, {\it An algebraic treatment of the Askey biorthogonal polynomials on the unit circle}, 
Forum of Mathematics, Sigma , {\bf 9} , 2021 , e68. arXiv:2102.01779v1.

\bi{VZ_JOPUC} L.Vinet and A.Zhedanov, {\it The CMV bispectrality of the Jacobi polynomials on the unit circle}, arXiv:2412.11031v1.

\bi{VZ_little-1} L.Vinet and A.Zhedanov, {\it A ``missing" family of classical orthogonal polynomials}, Journal of Physics A: Mathematical and Theoretical, {\bf 44}, 085201, 2011.

\bi{VZ_big-1} L.Vinet and A.Zhedanov, {\it A limit $q=-1$ for the big $q$-Jacobi polynomials}, Transactions of the American Mathematical Society, {\bf 364}, pp 5491-5507, 2012.

\eb

\end{document}